\documentclass[12pt]{amsart}
\usepackage{amsmath,amssymb}
\usepackage{amsthm}
\usepackage{amsfonts}
\usepackage{esint} 
\usepackage{mathrsfs}
\usepackage{enumerate}
\usepackage{geometry}
\usepackage[english]{babel}
\usepackage[colorlinks, linkcolor = blue, anchorcolor = blue, citecolor = blue]{hyperref}

\numberwithin{equation}{section}

\newtheorem{maintheorem}{Theorem}

\newtheorem{theorem}{Theorem}[section]
\newtheorem{definition}{Definition}[section]
\newtheorem{proposition}[theorem]{Proposition}
\newtheorem{corollary}[theorem]{Corollary}
\newtheorem{lemma}[theorem]{Lemma}
\newtheorem{remark}{Remark}[section]
\newtheorem{claim}[theorem]{Claim}

\newcommand{\vol}{\textup{vol}}

\newcommand{\Ric}{\textup{Ric}}
\newcommand{\D}{\mathcal{D}}

\newcommand{\ind}{\textup{ind}}
\newcommand{\R}{\mathcal{R}}

\newcommand{\id}{\textup{id}}
\newcommand{\T}{\mathbb{T}}

\newcommand{\pr}{\textup{pr}}

\begin{document}

\title{Scalar curvature rigidity for products of spheres and tori}

\author[T.-K.~A.~Chow]{Tsz-Kiu Aaron Chow}
\address{Department of Mathematics, Hong Kong University of Science and Technology, Hong Kong S.A.R., China}
%\email{chowtka@ust.hk}
\email{\href{chowtka@ust.hk}{chowtka@ust.hk}}

\maketitle

\begin{abstract}
We prove Llarull-type rigidity for $S^{n-m}\times\T^m$ ($3\le n\le 7$, $1\le m\le n-2$). If a closed spin $(M^n,g)$ admits a degree-nonzero map to $S^{n-m}\times\T^m$ whose spherical projection is area non-increasing, and there exists $\psi\in C^\infty(M)$ with $-\Delta_M\psi-\tfrac12|D_M\psi|^2+\tfrac12\big(R_M-(n-m)(n-m-1)\big)\ge0$, then $(M,g)$ is isometrically covered by $S^{n-m}\times\mathbb{R}^m$. For bands, we extend Gromov's torical inequality and obtain sharp width bounds: $\textup{dist}(\partial_-M,\partial_+M)\le 2\pi\sqrt{n/((n+1)\sigma)}$ when $R_M\ge (n-m)(n-m-1)+\sigma$. The method combines stable weighted slicing with a spectral Dirac operator argument.
\end{abstract}

\tableofcontents

\section{Introduction}
A classical conjecture of Geroch asserts that the $n$-torus $\T^n$ admits no Riemannian metric of positive scalar curvature. The conjecture (together with its rigidity case) was settled by Schoen--Yau for $3\le n\le 7$ via minimal hypersurfaces \cite{SY79}, and in all dimensions by Gromov--Lawson using the Dirac operator \cite{GL}. At the opposite end, Gromov asked whether a Riemannian metric $g$ on $S^n$ that strictly dominates the round metric $g_{S^n}$ must have scalar curvature strictly less than $n(n-1)$ somewhere. Llarull answered this by proving that if a closed spin $n$-manifold $(M^n,g)$ satisfies $R_M\ge n(n-1)$ and admits an area non-increasing map of nonzero degree to the round sphere, then rigidity holds in the strongest sense: the map is an isometry and hence $g$ is the pullback of the round metric \cite{Llarull}. Llarull's argument proceeds via spectral estimates for a twisted Dirac operator and index theoretic input tied to the nonzero degree. Subsequent work extended Llarull's theorem by lowering the regularity of the map \cite{Bar24, CHS24, LT22}, allowing targets beyond the sphere \cite{GS02, HSS24}, and treating manifolds with boundary under mean curvature assumptions \cite{BBHW24, CZ24, HLS, Lott21}.\\

These two results-—nonexistence and rigidity on tori, and Llarull-type rigidity on spheres--mark the extremes for scalar curvature on basic topological models. Between them lies the mixed geometry of products $S^{n-m}\times\T^m$. The product metric $g_{S^{n-m}}+ g_{\T^m}$ has scalar curvature exactly $(n-m)(n-m-1)$, the Llarull threshold for the spherical factor, yet it carries $m$ macroscopic flat directions encoding torus topology. Studying rigidity at this borderline illuminates how positive scalar curvature interacts with large-scale topology: it asks how much of the sphere’s rigidity persists once one permits $m$ flat directions, and conversely how enlargeability phenomena behind the torus obstruction constrain geometry when a positively curved factor is present. In this sense, $S^{n-m}\times\T^m$ is a natural ``midpoint" between the sphere and the torus, and scalar-curvature rigidity for degree–nonzero maps to this product probes the precise balance between curvature and topology.\\

In this paper we establish Llarull-type rigidity for degree-nonzero maps to $S^{n-m}\times\T^m$ for $3\leq n\leq 7$, together with quantitative band-width inequalities in the incomplete setting. Our main theorem asserts:

\begin{maintheorem}\label{main;thm;1}
	Let $3\leq n\leq 7$ and $1\leq m\leq n-2$. Let $M^n$ be a closed, orientable, connected spin manifold of dimension $n$. Let $g$ be a Riemannian metric on $M$. Let $\psi$ be a smooth function on $M$ such that	
	\[	-\Delta_M\psi -\frac{1}{2}|D_M\psi|^2 + \frac{1}{2}\Big(R_M - (n-m)(n-m-1)\Big) \geq 0.\]
Suppose that $\Phi:(M,g)\to (S^{n-m}\times \T^m, g_{S^{n-m}} + g_{\T^m})$ is a smooth map with the following properties:
	\begin{itemize}
		\item $\Phi$ has non-zero degree,
		\item $\pr_{S^{n-m}}\circ\Phi:(M,g)\to (S^{n-m}, g_{S^{n-m}})$ is area non-increasing when $n-m\geq 3$ and 1-Lipschitz when $n-m=2$.
	\end{itemize}
	Then $(M^n, g)$ is isometrically covered by $(S^{n-m}\times\mathbb{R}^m,\, g_{S^{n-m}} + g_{\mathbb{R}^m})$.
\end{maintheorem}

Our approach combines stable weighted slicing \cite{BHJ23} with a spectral Llarull-type argument compatible with the stability inequality. We construct a stable weighted slicing of order $m$, i.e., a nested family of hypersurfaces $\Sigma_m \subset \cdots \subset \Sigma_1 \subset M$
with positive weights that record the torus directions. On the bottom slice we use the stability inequality together with the Weitzenböck-Lichnerowicz identity for spinors. Equality forces $\Sigma_m$ to be a round $S^{n-m}$; we then propagate rigidity upward by foliating with weighted minimizers and applying a slice-by-slice splitting argument. It is worth noting that the local isometry statement in the case $n-m=3$ and $R_g\ge 6$ was proved by different methods in \cite{HLS}. Related consequences under higher mapping-degree assumptions was done in \cite{Tony}, which shows that when $\pr_{S^{n-m}}\circ\Phi:M^n\to S^{n-m}$ is a fiber bundle, it must be a Riemannian submersion.\\

In the incomplete setting, we extend Gromov's torical band inequality \cite{Gro18,Gro19} to bands over sphere--torus products. Bands are among the most flexible test-objects for scalar curvature: they detect how lower bounds on scalar curvature constrain the macroscopic separation of boundary components, yielding distance-type obstructions to fill-ins, doubling, and collaring. Recent progress has deepened the connection between scalar curvature and band width; see for instance \cite{CZ24, HKKZ23, HKKZ25}. For targets of the form $S^{n-m}\times \T^m$, the interplay between a positively curved spherical factor and $m$ flat directions raises a natural quantitative question: how far apart can two boundary components be kept under a scalar curvature lower bound that matches the spherical threshold up to a gap $\sigma>0$? Our results below give sharp-in-scale ($\sim\sigma^{-1/2}$) upper bounds, reflecting the model behavior of constant-curvature metrics and extending the torical width control to the mixed sphere-torus regime. In particular, requiring the spherical projection to be area non-increasing  prevents macroscopic stretching in the directions where scalar curvature is most constrained, and the band-width estimates quantify this restriction. To that end, we prove:\\

\begin{maintheorem}\label{main;thm;2}
	Let $3\leq n+1\leq 7$ and $1\leq m\leq n-2$. Let $M^{n+1}$ be an orientable, connected spin manifold of dimension $n$ with non-empty boundary with two connected components $\partial M = \partial_-M\sqcup\partial_+M$. Let $g$ be a Riemannian metric on $M$. Let $\sigma > 0$ be a positive real number. Let  $\Phi:(M,g)\to (S^{n-m}\times \T^m \times [-1,1], g_{S^{n-m}} + g_{\T^m} + dt^2)$ be a smooth map with the following properties:
	\begin{itemize}
		\item $\Phi$ has non-zero degree.
		\item $\Phi(\partial_{\pm}M)\subset S^{n-m}\times \T^m \times\{\pm1\}.$
		\item $\pr_{S^{n-m}}\circ\Phi:(M,g)\to (S^{n-m}, g_{S^{n-m}})$ is area non-increasing when $n-m\geq 3$ and 1-Lipschitz when $n-m=2$.
	\end{itemize}
The followings hold for the bandwidth of $M$:
\begin{enumerate}[(i)]
    \item Suppose that $R_M \geq  (n-m)(n-m-1) + \sigma$. Then
    \[	\textup{dist}_g(\partial_-M, \partial_+M) \leq  2\pi\sqrt{\frac{n}{(n+1)\sigma}}.\]
    \item Suppose that $\psi$ is a smooth function on $M$ such that	
	\[	-\Delta_M\psi -\frac{1}{2}|D_M\psi|^2 + \frac{1}{2}\Big(R_M - (n-m)(n-m-1) - \sigma \Big) \geq 0.\]
    Then
    \[	d(\partial_-M, \partial_+M) \leq  \frac{2\pi}{\sqrt{\sigma}}.\]
\end{enumerate}
\end{maintheorem}

\bigskip

\begin{remark}
For the sake of exposition we will henceforth work under the stronger hypothesis that
    \[\pr_{S^{n-m}}\circ\Phi:(M,g)\to (S^{n-m}, g_{S^{n-m}})\]
is 1-Lipschitz. This assumption streamlines the notations used in the proof. All arguments, however, carry over verbatim when $n-m\geq 3$ if one merely assumes that the map is area non-increasing.
\end{remark}

\section{Stable weighted slicing}

The purpose of this section is two-fold: In Proposition~\ref{slicing;exist;prop} we recall that the topological data coming from a 1-Lipschitz map of non-zero degree forces the existence of a stable weighted slicing of order $m$; In Proposition \ref{slicing;prop1} we derive an integral inequality on each slice that will later serve as the spectral substitute for the scalar curvature lower bound, when we apply Llarull type argument to study the bottom slice in Section 3. 
We begin by recalling the notion of stable weighted slicing as in \cite{BHJ23}.

\begin{definition}\label{slicing;def}[Stable weighted slicing of order $m$]
Let $1\leq m\leq n-2$. Let $(M^n, g)$ be a closed, connected, orientable Riemannian manifold of dimension $n$.

 A stable weighted slicing of order $m$ consists of a collection of closed, connected, orientable submanifolds $\Sigma_k$, $k\in\{0,1,\dots, m\}$,  a collection of positive functions $u_k\in C^{\infty}(\Sigma_k), k\in\{1,\dots, m\}$, and a collection of positive functions $\rho_k\in C^{\infty}(\Sigma_k), k\in\{0,1,\dots, m\}$  with the following properties:
 	\begin{enumerate}[(i)]
		\item $(\Sigma_0, \rho_0) = (M, e^{\psi})$ .
		\item For each $k\in \{0,1,\dots, m\}$, we have $\dim \Sigma_k = n-k$.
		\item For each $k\in \{1,\dots, m\}$, $\Sigma_k$ is a closed, connected, embedded, orientable hypersurface in $\Sigma_{k-1}$. Moreover, $\Sigma_k$ is a stable critical point of the $\rho_{k-1}$-weighted area
				\[	\mathcal{H}_{\rho_{k-1}}^{n-k}(\Sigma) = \int_{\Sigma}\rho_{k-1} d\mu\]
				is the class of hypersurfaces $\Sigma\subset \Sigma_{k-1}$.
		\item For each $k\in \{1,\dots, m\}$, the function $u_k\in C^{\infty}(\Sigma_k)$ is a first eigenfunction of the stability operator associated with the $\rho_{k-1}$-weighted area.
		\item For each $k\in \{1,\dots, m\}$, the function $\rho_k\in C^{\infty}(\Sigma_k)$ is given by  $\rho_k = \rho_{k-1}|_{\Sigma_k}\cdot u_k$.
	\end{enumerate}
\end{definition}

Now, we establish the existence of a stable weighted slicing under the topological assumptions in Theorem \ref{main;thm;1}, in a form tailored to the sphere-torus product.

\begin{proposition}\label{slicing;exist;prop}
	Let $2\leq n\leq 7$ and $1\leq m\leq n-2$. Let $(M^n, g)$ be a closed, connected Riemannian manifold of dimension $n$ with a smooth map $\Phi:(M^n,g)\to (S^{n-m}\times \T^m, g_{S^{n-m}} + g_{\mathbb{T}^{m}})$ of non-zero degree such that $\pr_{S^{n-m}}\circ\Phi :(M^n,g)\to (S^{n-m}, g_{S^{n-m}})$ is 1-Lipschitz. Then there exists a stable weighted slicing
		\[	\Sigma_m\subset\Sigma_{m-1}\subset\cdots\subset\Sigma_1\subset\Sigma_0 = M^n\]
		 of order $m$.
		 
		  In addition, there exists a collection of smooth maps $\Phi_k: (\Sigma_k, g|_{\Sigma_k})\to (S^{n-m}\times\mathbb{T}^{m-k}, g_{S^{n-m}} + g_{\mathbb{T}^{m-k}})$, $k\in\{0,1,\dots, m\}$, such that each $\Phi_k$ has non-zero degree and $\pr_{S^{n-m}}\circ\Phi_k :(\Sigma_k,g|_{\Sigma_k})\to (S^{n-m}, g_{S^{n-m}})$ is 1-Lipschitz.
\end{proposition}

\begin{proof}
	The existence of stable weighted slicings is given in \cite[Theorem 1.5]{BHJ23}, it suffices to construct maps $\Phi_k$ which satisfy the assertions. We first recall the set-up in \cite[Theorem 1.5]{BHJ23}. Denote the projection of $\Phi$ onto the factors by $\phi_0:M\to S^{n-m}$ and $\phi_1,\dots, \phi_m: M\to S^1$. By assumption, the map $\phi_0 = \pr_{S^{n-m}}\circ\Phi$ is 1-Lipschitz. Let $\Theta$ be a top-form of $S^{n-m}$ such that $\int_{S^{n-m}}\Theta = 1$, and $\theta$ be a one-form of $S^1$ such that $\int_{S^1}\theta = 1$.  Define the pull-back forms $\Omega = \phi_0^*\Theta$ and $\omega_j = \phi_j^*\theta$. From the proof of \cite[Theorem 1.5]{BHJ23}, for each $k\in \{1,\dots, m\}$, the slice $\Sigma_k$ is a smooth hypersurface in $\Sigma_{k-1}$ such that 
		\[	\int_{\Sigma_k}\omega_{k+1}\wedge\cdots\wedge\omega_m\wedge\Omega = \deg(\Phi).\]
				
		Now, we define the smooth maps $\Phi_k: (\Sigma_k, g|_{\Sigma_k})\to (S^{n-m}\times\mathbb{T}^{m-k}, g_{S^{n-m}} + g_{\mathbb{T}^{m-k}})$ by 
		\begin{align*}
			\begin{cases}
				&\Phi_k := \left(\phi_0|_{\Sigma_k},\phi_{k+1}|_{\Sigma_k},\dots,  \phi_{m}|_{\Sigma_k}\right),\quad \text{if}\quad k\leq m-1,\\
				&\Phi_m := \phi_0|_{\Sigma_m}.
			\end{cases}
		\end{align*}
		It follows from the definition that all $\pr_{S^{n-m}}\circ\Phi_k = \phi_0|_{\Sigma_k}$ are 1-Lipschitz. Moreover,
		\begin{align*}
			\deg(\Phi_k) = \int_{\Sigma_k}\omega_{k+1}\wedge\cdots\wedge\omega_m\wedge\Omega = \deg(\Phi) \neq 0.
		\end{align*}
\end{proof}

\medskip
With a stable weighted slicing in hand we now turn to quantitative properties on each slice. The following lemma records the weighted stability inequality that underpins curvature comparisons later in the paper.

\begin{lemma}\label{slicing;lem1}
	Suppose that $\Sigma_k$ is a stable critical point of the $\rho_{k-1}$-weighted area
				\[	\mathcal{H}_{\rho_{k-1}}^{n-k}(\Sigma) = \int_{\Sigma}\rho_{k-1} d\mu\]
				is the class of hypersurfaces $\Sigma\subset \Sigma_{k-1}$, then
	\begin{align*}
	0 \leq& 	\int_{\Sigma_k}\rho_{k-1}|D_{\Sigma_k}f|^2  - \frac{1}{2}\int_{\Sigma_k}(R_{\Sigma_{k-1}} - R_{\Sigma_k} + |A_{\Sigma_k}|^2)\rho_{k-1}f^2\\
		&\notag + \int_{\Sigma_k} \left(\Delta_{\Sigma_{k-1}}\log\rho_{k-1} + \frac{1}{2} |D_{\Sigma_{k-1}}\log\rho_{k-1}|^2
		 \right) \rho_{k-1}f^2\\
		&\notag - \int_{\Sigma_k} \left(\Delta_{\Sigma_{k}}\log\rho_{k-1} + \frac{1}{2} |D_{\Sigma_{k}}\log\rho_{k-1}|^2
		 \right) \rho_{k-1}f^2
	\end{align*}
for all $f\in C^{\infty}(\Sigma_k)$.
\end{lemma}

\begin{proof}
	The second variation of weighted area \cite{BHJ23} gives
	
	\begin{align*}
		0 \leq &-\int_{\Sigma_k}\rho_{k-1}f\Delta_{\Sigma_k}f -\int_{\Sigma_k}(|A_{\Sigma_k}|^2 + \Ric_{\Sigma_{k-1}}(\nu_{\Sigma_k}, \nu_{\Sigma_k}))\rho_{k-1}f^2\\
		& + \int_{\Sigma_k}(D_{\Sigma_{k-1}}^2\log \rho_{k-1})(\nu_{\Sigma_k}, \nu_{\Sigma_k})\rho_{k-1}f^2 - \int_{\Sigma_k}\langle D_{\Sigma_k}\log\rho_{k-1},\, D_{\Sigma_k}f\rangle\rho_{k-1}f
	\end{align*}
for all $f\in C^{\infty}(\Sigma_k)$. By the Gauss equation, 

	\begin{align*}
		|A_{\Sigma_k}|^2 + \Ric_{\Sigma_{k-1}}(\nu_{\Sigma_k}, \nu_{\Sigma_k}) = \frac{1}{2}(R_{\Sigma_{k-1}} - R_{\Sigma_k} + |A_{\Sigma_k}|^2 + H_{\Sigma_k}^2).
	\end{align*}
Moreover, we have the identity
	
	\begin{align*}
		&(D_{\Sigma_{k-1}}^2\log \rho_{k-1})(\nu_{\Sigma_k}, \nu_{\Sigma_k})\\ &= \Delta_{\Sigma_{k-1}}\log\rho_{k-1} - \Delta_{\Sigma_{k}}\log\rho_{k-1}
		 - H_{\Sigma_k} \langle D_{\Sigma_{k-1}}\log\rho_{k-1},\, \nu_{\Sigma_k}\rangle.
	\end{align*}
Adding the above two identities to the second variation formula, we obtain

\begin{align*}
		0 \leq &-\int_{\Sigma_k}\rho_{k-1}f\Delta_{\Sigma_k}f - \frac{1}{2}\int_{\Sigma_k}(R_{\Sigma_{k-1}} - R_{\Sigma_k} + |A_{\Sigma_k}|^2)\rho_{k-1}f^2\\
		& + \int_{\Sigma_k} (\Delta_{\Sigma_{k-1}}\log\rho_{k-1} - \Delta_{\Sigma_{k}}\log\rho_{k-1}
		 ) \rho_{k-1}f^2 - \int_{\Sigma_k}\langle D_{\Sigma_k}\log\rho_{k-1},\, D_{\Sigma_k}f\rangle\rho_{k-1}f \\
		&  - \frac{1}{2}\int_{\Sigma_k}H_{\Sigma_k}^2\rho_{k-1}f^2 - \int_{\Sigma_k}H_{\Sigma_k} \langle D_{\Sigma_{k-1}}\log\rho_{k-1},\, \nu_{\Sigma_k}\rangle \rho_{k-1}f^2
\end{align*}
for all $f\in C^{\infty}(\Sigma_k)$. Since $\Sigma_k$ is a critical point of the $\rho_{k-1}$-weighted area, 

\[	H_{\Sigma_k} = -\langle D_{\Sigma_{k-1}} \log\rho_{k-1},\, \nu_{\Sigma_k}\rangle.\]
Subsequently,

\begin{align}\label{slicing;eqn1}
			0 \leq &-\int_{\Sigma_k}\rho_{k-1}f\Delta_{\Sigma_k}f - \frac{1}{2}\int_{\Sigma_k}(R_{\Sigma_{k-1}} - R_{\Sigma_k} + |A_{\Sigma_k}|^2)\rho_{k-1}f^2\\
		&\notag + \int_{\Sigma_k} (\Delta_{\Sigma_{k-1}}\log\rho_{k-1} - \Delta_{\Sigma_{k}}\log\rho_{k-1}
		 ) \rho_{k-1}f^2 - \int_{\Sigma_k}\langle D_{\Sigma_k}\log\rho_{k-1},\, D_{\Sigma_k}f\rangle\rho_{k-1}f \\
		&\notag  + \frac{1}{2}\int_{\Sigma_k} \langle D_{\Sigma_{k-1}}\log\rho_{k-1},\, \nu_{\Sigma_k}\rangle^2 \rho_{k-1}f^2\\
		= &\notag  \int_{\Sigma_k}\rho_{k-1}|D_{\Sigma_k}f|^2  - \frac{1}{2}\int_{\Sigma_k}(R_{\Sigma_{k-1}} - R_{\Sigma_k} + |A_{\Sigma_k}|^2)\rho_{k-1}f^2\\
		&\notag + \int_{\Sigma_k} (\Delta_{\Sigma_{k-1}}\log\rho_{k-1} - \Delta_{\Sigma_{k}}\log\rho_{k-1}
		 ) \rho_{k-1}f^2 + \frac{1}{2}\int_{\Sigma_k} \langle D_{\Sigma_{k-1}}\log\rho_{k-1},\, \nu_{\Sigma_k}\rangle^2 \rho_{k-1}f^2\\
		= &\notag  \int_{\Sigma_k}\rho_{k-1}|D_{\Sigma_k}f|^2  - \frac{1}{2}\int_{\Sigma_k}(R_{\Sigma_{k-1}} - R_{\Sigma_k} + |A_{\Sigma_k}|^2)\rho_{k-1}f^2\\
		&\notag + \int_{\Sigma_k} \left(\Delta_{\Sigma_{k-1}}\log\rho_{k-1} + \frac{1}{2} |D_{\Sigma_{k-1}}\log\rho_{k-1}|^2
		 \right) \rho_{k-1}f^2\\
		&\notag - \int_{\Sigma_k} \left(\Delta_{\Sigma_{k}}\log\rho_{k-1} + \frac{1}{2} |D_{\Sigma_{k}}\log\rho_{k-1}|^2
		 \right) \rho_{k-1}f^2
\end{align}
for all $f\in C^{\infty}(\Sigma_k)$, where in the last step we have used the identity

\[	 \langle D_{\Sigma_{k-1}}\log\rho_{k-1},\, \nu_{\Sigma_k}\rangle^2 = |D_{\Sigma_{k-1}}\log\rho_{k-1}|^2 - |D_{\Sigma_{k}}\log\rho_{k-1}|^2.\]

\end{proof}

\bigskip

\begin{lemma}[c.f. \cite{BH24, SY79, SY17}]\label{slicing;lem2}
	For each $k\in \{1,\dots, m\}$, the function $\rho_k$ satisfies the inequality
	
	\begin{align*}
	&\Delta_{\Sigma_k}\log\rho_k + \frac{1}{2} |D_{\Sigma_{k}}\log\rho_{k}|^2	\\
	\leq\, & \Delta_{\Sigma_{k-1}}\log\rho_{k-1} + \frac{1}{2} |D_{\Sigma_{k-1}}\log\rho_{k-1}|^2 -\frac{1}{2}(R_{\Sigma_{k-1}} - R_{\Sigma_k} + |A_{\Sigma_k}|^2) - \frac{1}{2} |D_{\Sigma_{k}}\log u_{k}|^2.
	\end{align*}

\end{lemma}

\begin{corollary}[c.f. \cite{BH24}]\label{slicing;cor1}
	For each $k\in \{1,\dots, m\}$, the function $\rho_k$ satisfies the inequality
	
	\begin{align*}
	&\Delta_{\Sigma_k}\log\rho_k + \frac{1}{2} |D_{\Sigma_{k}}\log\rho_{k}|^2	\\
	\leq\, &  \Delta_{M}\log\rho_{0} + \frac{1}{2} |D_{\Sigma_{M}}\log\rho_{0}|^2-\frac{1}{2}(R_{M} - R_{\Sigma_k} ) - \frac{1}{2} \sum_{j=1}^k|D_{\Sigma_{j}}\log u_{j}|^2  - \frac{1}{2} \sum_{j=1}^k|A_{\Sigma_j}|^2.
	\end{align*}
\end{corollary}

Combining Lemma \ref{slicing;lem1} and Corollary \ref{slicing;cor1}, we thus obtain

\begin{proposition}\label{slicing;prop1}
	Let $2\leq n\leq 7$ and $1\leq m\leq n-1$. Let $M^n$ be a closed, connected spin manifold of dimension $n$. Let $g$ be a Riemannian metric on $M^n$ and 
    \[	\Sigma_m\subset \cdots\subset \Sigma_1\subset\Sigma_0 = M^n\]
    is a stable weighted slicing of order $m$. Then for each slice of order $k\in\{1,\dots,m\}$, we have
	\begin{align}\label{slicing;eqn3}
	0 \leq& 	\int_{\Sigma_k}\rho_{k-1}|D_{\Sigma_k}f|^2  - \frac{1}{2}\int_{\Sigma_k}(R_{M} - R_{\Sigma_k} + |A_{\Sigma_k}|^2)\rho_{k-1}f^2\\
	&\notag + \int_{\Sigma_k} \left(\Delta_{\Sigma_{0}}\log\rho_{0} + \frac{1}{2} |D_{\Sigma_{0}}\log\rho_{0}|^2
		 \right) \rho_{k-1}f^2\\
		&\notag - \int_{\Sigma_k} \left(\Delta_{\Sigma_{k}}\log\rho_{k-1} + \frac{1}{2} |D_{\Sigma_{k}}\log\rho_{k-1}|^2
		 \right) \rho_{k-1}f^2
\end{align}
for all $f\in C^{\infty}(\Sigma_k)$.
\end{proposition}

\begin{corollary}\label{slicing;cor2}
    Assume the hypothesss of Theorem~\ref{main;thm;1} and let 
        \[	\Sigma_m\subset \cdots\subset \Sigma_1\subset\Sigma_0 = M^n\]
    be a stable weighted slicing of order $m$. Then for each slice of order $k\in\{1,\dots,m\}$, we have
	\begin{align}\label{slicing;eqn4}
	0 \leq& 	\int_{\Sigma_k}\rho_{k-1}|D_{\Sigma_k}f|^2  - \frac{1}{2}\int_{\Sigma_k} \left(  (n-m)(n-m-1) -  R_{\Sigma_k}\right) \rho_{k-1}f^2\\
		&\notag - \int_{\Sigma_k} \left(\Delta_{\Sigma_{k}}\log\rho_{k-1} + \frac{1}{2} |D_{\Sigma_{k}}\log\rho_{k-1} |^2
		 \right) \rho_{k-1}f^2
\end{align}
for all $f\in C^{\infty}(\Sigma_k)$.
\end{corollary}

\begin{proof}
This follows from putting the requirements that $\rho_0 = e^{\psi}$ and 
\[	-\Delta_M\psi -\frac{1}{2}|D_M\psi|^2 + \frac{1}{2}\Big(R_M - (n-m)(n-m-1)\Big) \geq 0\]
into Proposition~\ref{slicing;prop1}.
\end{proof}

\bigskip

\section{A spectral Llarull's theorem}\label{spectral;Llarull;section}

This section aims to extend Llarull's argument to the spectral setting which arises from the stability of weighted slicing. To that end, we will prove

\begin{theorem}\label{spectral;Llarull;thm}
	Let $N^n$ be a closed, connected spin manifold of dimension $n$. Let $g$ be a Riemannian metric on $N$. Suppose that there exists a smooth function $\rho >0$ on $N$ such that
\begin{align}\label{spectral;Llarull;assumption}
	0 \leq& 	\int_{N}\rho|D_{N}f|^2  - \frac{1}{2}\int_{N}(n(n-1) - R_{N})\rho\, f^2 - \int_{N} \left(\Delta_{N}\log\rho + \frac{1}{2} |D_{N}\log\rho|^2
		 \right) \rho\, f^2
\end{align}
for all $f\in C^{\infty}(N)$.
Suppose that $\Phi:(N,g)\to (S^n, g_{S^n})$ is a smooth map with the following properties:
	\begin{itemize}
		\item $\Phi$ has non-zero degree,
		\item $\Phi$ is 1-Lipschitz.
	\end{itemize}
	Then $\rho$ is a constant on $N$ and $\Phi$ is a Riemannian isometry.
\end{theorem}

\begin{remark}
For $n\geq 3$ the conclusion still holds if the 1-Lipschitz condition on $\Phi$ is replaced by the weaker assumption that $\Phi$ is area non-increasing; the argument in the sequel of this section below goes through unchanged under this milder assumption. 
\end{remark}

\begin{lemma}\label{Llarull;lem1}
	Let $N$ be a closed, spin manifold and $\rho>0$ a smooth function on $N$. Suppose that $\varphi\in C^{\infty}(N)$ is a smooth function, then
	\begin{align}\label{key;estimate}
	&\int_{N}\rho\,|D_{N}(\rho^{-1/2}\varphi)|^2 - 	\int_{N} \left(\Delta_{N}\log\rho + \frac{1}{2} |D_{N}\log\rho|^2
		 \right)\varphi^2\\
		 \notag \leq &\, -\frac{1}{2(n+1)}\int_N |D_N\log\rho|^2\varphi^2 +  \frac{2n}{n-1}\int_{N}|D_{N}\varphi|^2.
	\end{align}

\end{lemma}

\begin{proof}
We calculate
	\[	D_{N}(\rho^{-1/2}\varphi) = -\frac{1}{2}\rho^{-3/2}\varphi D_{N}\rho  + \rho^{-1/2} D_{N}\varphi,\]
	this gives
	\begin{align*}
		&\rho|D_{N}(\rho^{-1/2}\varphi)|^2\\
		=\, & \frac{1}{4}\rho^{-2}|D_{N}\rho|^2 \varphi^2 -\rho^{-1}\langle D_{N}\rho,\, D_{N}\varphi\rangle \varphi + |D_{N}\varphi|^2\\
		=\, & \frac{1}{4}|D_{N}\log\rho|^2\varphi^2 - \langle D_{N}\log\rho,\, D_{N}\varphi\rangle\varphi + |D_{N}\varphi|^2.
	\end{align*}
On the other hand, integrating by parts gives

\begin{align*}
	\int_{N} (\Delta_{N}\log\rho) \varphi^2 &= -\int_{N}\langle D_{N}\log\rho,\, D_{N}\varphi^2\rangle \\
	&= - 2\int_{N}\langle D_{N}\log\rho,\, D_{N}\varphi\rangle \varphi.
\end{align*}
Putting the two identities together, we get
	\begin{align*}
	&\int_{N}\rho\,|D_{N}(\rho^{-1/2}\varphi)|^2 - 	\int_{N} \left(\Delta_{N}\log\rho + \frac{1}{2} |D_{N}\log\rho|^2
		 \right)\varphi^2\\
		 =\, & \int_{N} \left( \frac{1}{4}|D_{N}\log\rho|^2\varphi^2 - \langle D_{N}\log\rho,\, D_{N}\varphi\rangle\varphi + |D_{N}\varphi|^2 \right)\\
		 &\quad  + \int_{N} 2\langle D_{N}\log\rho,\, D_{N}\varphi\rangle \varphi - \int_{N} \frac{1}{2} |D_{N}\log\rho|^2\varphi^2 \\
		 =\, & -\frac{1}{4}\int_{N}|D_{N}\log\rho|^2\varphi^2 +   \int_{N} \langle D_{N}\log\rho,\, D_{N}\varphi\rangle \varphi + \int_{N}|D_{N}\varphi|^2.
	\end{align*}
Next, using the Young's inequality,
\begin{align*}
	\langle D_{N}\log\rho,\, D_{N}\varphi\rangle \varphi &\leq \frac{n-1}{4(n+1)}|D_N\log\rho||\varphi|^2 + \frac{n+1}{n-1}|D_N\varphi|^2. 
\end{align*}
Putting everything together, the Lemma follows.

\end{proof}

\subsection{Proof of Theorem \ref{spectral;Llarull;thm} when $n$ is even}
Choose a spin structure on $N$ and let $S$ be the spinor bundle over $N$. Let $E_0$ be the spinor bundle of the round sphere $S^n$. Since $n$ is even, we have the splittings $S = S^+\oplus S^-$ and  $E_0 = E_0^+\oplus E_0^-$. We consider the twisted bundles
	\[	E = (S^+\otimes\Phi^*E_0^+)\oplus (S^-\otimes\Phi^*E_0^-),\]
	\[	F = (S^+\otimes\Phi^*E_0^-)\oplus (S^-\otimes\Phi^*E_0^+).\]
We then consider the twisted Dirac operators
	\[	\D_+: H^1(N, S^+\otimes\Phi^*E_0^+) \to L^2(N, S^-\otimes\Phi^*E_0^+),\]
	\[	\D_-: H^1(N, S^-\otimes\Phi^*E_0^-) \to L^2(N, S^+\otimes\Phi^*E_0^-).\]
Then the twisted Dirac operator
	\begin{align*}
	\D^{E} = 
		\begin{pmatrix}
			0 & \D_-\\
			\D_+ & 0
		\end{pmatrix}
	\end{align*}
maps sections of $E$ to sections of $F$. From the proof of \cite[Proposition 2.2]{BBCH24}, we have
	\[	\ind(\D^{E}) = 2\deg(\Phi).\]
Up to switching $\D^{E}$ with its adjoint, we may assume that $\ind(\D^{E}) > 0$.  Hence we can find a non-trivial spinor field $s\in E$ such that $\D^Es=0$. For $\epsilon > 0$, denote by $\varphi_{\epsilon} = (|s|^2 + \epsilon)^{1/2}$. Taking $ f = \rho^{-1/2}\varphi_\epsilon$ into the assumption (\ref{spectral;Llarull;assumption}), we get

\begin{align}\label{Llarull;eqn1}
		0 \leq& 	\int_{N}\rho|D_{N}(\rho^{-1/2}\varphi_\epsilon)|^2  - \frac{1}{2}\int_{N}(n(n-1) - R_{N})\varphi_\epsilon^2 - \int_{N} \left(\Delta_{N}\log\rho + \frac{1}{2} |D_{N}\log\rho|^2
		 \right)\varphi_\epsilon^2. 
\end{align}
Combining Lemma \ref{Llarull;lem1} with (\ref{Llarull;eqn1}), we have

\begin{align}\label{Llarull;eqn2}
	0 \leq &  -\frac{1}{4(n+1)}\int_N |D_N\log\rho|^2\varphi_\epsilon^2 +  \frac{n}{n-1}\int_{N} |D_N\varphi_\epsilon|^2 - \frac{1}{4}\int_{N}(n(n-1) - R_{N})\varphi_\epsilon^2.
\end{align}

\begin{claim}[a refined Kato's inequality]\label{refined;Kato;ineq}
	We have
	\[	|D_N\varphi_\epsilon|^2 \leq \frac{n-1}{n}|\nabla^Es|^2.\]
\end{claim}

\begin{proof}[Proof of Claim~\ref{refined;Kato;ineq}]
	Fix a point $p\in N$. If $|s|(p) = 0$, the inequality is trivial. It suffices to consider $|s|(p)\neq 0$. Choose an orthonormal frame $\{e_1,\dots, e_n\}$ around $p$ such that $D_{e_1}\varphi_\epsilon = |D_N\varphi_\epsilon|$. Now, since $\D^Es = 0$, observe that
		\[	\nabla^E_{e_1}s = \sum_{j=2}^n e_1\cdot e_j\cdot\nabla^E_{e_j}s.\]
		Using Cauchy-Schwarz inequality,
		\begin{align*}
			|\nabla_{e_1}^Es|^2 &\leq (n-1)\sum_{j=2}^n|\nabla_{e_j}^E s|^2,
		\end{align*}
		which implies
		\[	|\nabla_{e_1}^Es|^2 \leq \frac{n-1}{n}|\nabla^Es|^2.\]
		Therefore,
		\begin{align*}
		|D_N\varphi_\epsilon|^2 = |D_{e_1}\varphi_\epsilon|^2 \leq \frac{|s|^2}{|s|^2+\epsilon}|\nabla_{e_1}^Es|^2 \leq 	\frac{|s|^2}{|s|^2+\epsilon}\cdot \frac{n-1}{n}|\nabla^Es|^2.
		\end{align*}
This proves the claim.
\end{proof}

Combining Claim~\ref{refined;Kato;ineq} with (\ref{Llarull;eqn2}), we thus obtain
\begin{align*}
		0 &\leq -\frac{1}{4(n+1)}\int_N |D_N\log\rho|^2|s|^2 + \int_N|\nabla^Es|^2 -\frac{1}{4}\int_{N}(n(n-1) - R_{N})|s|^2\\
		 &\quad -\frac{\epsilon}{4(n+1)}\int_N |D_N\log\rho|^2 -\frac{\epsilon}{4}\int_{N}(n(n-1) - R_{N}).
\end{align*}
Since this holds for all $\epsilon > 0$, taking limit $\epsilon\to 0$ we finally obtain
\begin{align}\label{Llarull;eqn3}
		0 \leq -\frac{1}{4(n+1)}\int_N |D_N\log\rho|^2|s|^2 + \int_N|\nabla^Es|^2 -\frac{1}{4}\int_{N}(n(n-1) - R_{N})|s|^2.
\end{align}
On the other hand, the Weitzenb\"ock identity gives

	\begin{align}\label{Llarull;eqn4}
	\int_{N}|\nabla^E s|^2 = \int_{N}|\D^Es|^2 - \frac{1}{4}\int_{N}R_{N}|s|^2 - \int_{N}\langle \R^Es,\, s\rangle.
	\end{align}
Since $\D^Es = 0$, upon combining (\ref{Llarull;eqn3}) with (\ref{Llarull;eqn4}), we get
\begin{align}\label{Llarull;eqn5}
	0 \leq -\frac{1}{4(n+1)}\int_N |D_N\log\rho|^2|s|^2 -\int_{N}\left( \frac{n(n-1)}{4}|s|^2 + \langle \R^Es,\, s\rangle\right).
\end{align}
By the calculation of Llarull \cite{Llarull} (also see \cite[Proposition A.1]{BBHW24}), we have
	\begin{align}\label{Llarull;eqn6}
		\langle \R^Es,\, s\rangle\geq -\frac{1}{4}\sum_{\substack{1\leq j,k\leq n\\j\neq k}}\mu_j\mu_k |s|^2,
	\end{align}
where $\mu_1,\dots,\mu_{n}\geq 0$ denote the singular values of the differential $d\Phi_x: (T_xN, g_x)\to (T_{\Phi(x)}S^{n}, g_{S^{n}})$. Putting (\ref{Llarull;eqn5}) and (\ref{Llarull;eqn6}) together,
\begin{align}\label{Llarull;eqn7}
	0 \leq -\frac{1}{4(n+1)}\int_N |D_N\log\rho|^2|s|^2 -\int_{N} \sum_{\substack{1\leq j,k\leq n\\j\neq k}}(1-\mu_j\mu_k)|s|^2.
\end{align}
We thus deduce from (\ref{Llarull;eqn7}) that $D_N\log\rho \equiv 0$ and $\mu_j \equiv 1$ on $N$. Consequently, $\Phi$ is an isometry.

\bigskip

\subsection{Proof of Theorem \ref{spectral;Llarull;thm} when $n$ is odd}
Following the modification made in \cite{Llarull}, we consider the product $\tilde{N} = N\times S^1_r$ equipped with product metric $g_{\tilde{N}} = g_N + r^2 g_{S^1}$, where $r>1$ is a large number . Consider
	\[	N\times S^1_r \xrightarrow{\Phi\times \frac{1}{r}\id} S^n\times S^1\xrightarrow{h} S^{n+1}, \]
	where $h$ is a 1-Lipschitz suspension map of degree one. The composition $\Phi_r = h\circ(\Phi\times \frac{1}{r}\id): \tilde{N}\to S^{n+1}$ is 1-Lipschitz and has $\deg(\Phi_r)\neq 0$.
	
	Choose a spin structure on $\tilde{N}$, and let $\tilde{S}$ be the spinor bundle over $\tilde{N}$. Let $E_0$ be the spinor bundle over the round $S^{n+1}$.  Since $n+1$ is even,  we have the splittings $\tilde{S} = \tilde{S}^+\oplus \tilde{S}^-$ and  $E_0 = E_0^+\oplus E_0^-$. We consider the twisted bundles
	\[	\tilde{E} = (\tilde{S}^+\otimes\Phi^*E_0^+)\oplus (\tilde{S}^-\otimes\Phi^*E_0^-),\]
	\[	\tilde{F} = (\tilde{S}^+\otimes\Phi^*E_0^-)\oplus (\tilde{S}^-\otimes\Phi^*E_0^+).\]
We then consider the twisted Dirac operators
	\[	\tilde{\D}_+: H^1(N, \tilde{S}^+\otimes\Phi^*E_0^+) \to L^2(N, \tilde{S}^-\otimes\Phi^*E_0^+),\]
	\[	\tilde{\D}_-: H^1(N, \tilde{S}^-\otimes\Phi^*E_0^-) \to L^2(N, \tilde{S}^+\otimes\Phi^*E_0^-).\]
Then the twisted Dirac operator
	\begin{align*}
	\D^{\tilde{E}} = 
		\begin{pmatrix}
			0 & \tilde{\D}_-\\
			\tilde{\D}_+ & 0
		\end{pmatrix}
	\end{align*}
maps sections of $\tilde{E}$ to sections of $\tilde{F}$. From the proof of \cite[Proposition 2.2]{BBCH24}, we have
	\[	\ind(\D^{\tilde{E}}) = 2\deg(\Phi_r).\]
Up to switching $\D^{\tilde{E}}$ with its adjoint, we may assume that $\ind(\D^{\tilde{E}}) > 0$.  Hence we can find a non-trivial spinor field $s\in \tilde{E}$ such that $\D^{\tilde{E}}s=0$.

Next, choose an orthonormal frame $\{e_1,\dots, e_n, e_{n+1}\}$ around $(x,t)\in \tilde{N}= N\times S^1_r$ such that $\{e_1,\dots, e_n\}$ is an orthonormal frame around $x\in N$, and an orthonormal frame $\{f_1,\dots, f_{n+1}\}$ around $\Phi_r(x,t)\in S^{n+1}$. Let $\tilde{\mu}_1,\dots,\tilde{\mu}_{n+1}\geq 0$ the singular values of the differential $d(\Phi_r)_{(x,t)}: (T_{(x,t)}\tilde{N}, \tilde{g}_{(x,t)})\to (T_{\Phi_r(x,t)}S^{n+1}, g_{S^{n+1}})$  so that $d\Phi_r(e_k) = \tilde{\mu}_k f_k$.  The calculation in \cite[page 68]{Llarull} gives
\begin{align}\label{Llarull;eqn8}
			\langle \R^Es,\, s\rangle \geq -\frac{1}{4}\sum_{\substack{1\leq j,k\leq n\\j\neq k}}\tilde{\mu}_j\tilde{\mu}_k |s|^2 - \frac{n}{2r}|s|^2.
\end{align}
Repeating the same argument as in the even dimensional case, we thus obtain
\begin{align}\label{Llarull;eqn9}
	0 \leq -\frac{1}{4(n+1)}\int_N |D_N\log\rho|^2|s|^2 -\int_{\tilde{N}} \sum_{\substack{1\leq j,k\leq n\\j\neq k}}(1-\tilde{\mu}_j\tilde{\mu}_k)|s|^2 - \int_{\tilde{N}} \frac{2n}{r}|s|^2.
\end{align}
This implies $D_N\log\rho = 0$. Now, if $\tilde{\mu}_j < 1$ for some $j\in\{1,\dots,n\}$ at some point $p\in N$, we can choose $r$ sufficiently large such that (\ref{Llarull;eqn9}) is violated. This implies $\tilde{\mu}_j \equiv 1$. Next, since $d\Phi_r = dh \circ (d\Phi\oplus \frac{1}{r}\id_{TS^1})$ and $|\wedge^2 d\Phi|\leq 1$, we have
\begin{align*}
	1 = \tilde{\mu}_j\tilde{\mu}_k \leq \|dh\||d\Phi(e_j)\wedge d\Phi(e_k)|\leq |d\Phi(e_j)\wedge d\Phi(e_k)|\leq 1.
\end{align*}
Consequently $\Phi$ is an isometry.
\bigskip

\section{Foliations arising from the equality case}\label{foliation;section}

In this section we analyze what happens when the slice $\Sigma_k$ satisfies the equality $R_{\Sigma_k} = (n-m)(n-m-1)$. This would ultimately lead to local isometry. We will adapt the idea of \cite{CKL24, Zhu20}, to show that if $R_{\Sigma_k} = (n-m)(n-m-1)$, then the submanifold $\Sigma_k$ is totally geodesic in $\Sigma_{k-1}$ and admits a local foliation $\{\Sigma_{k,t}\}_{t\in (-\epsilon, \epsilon)}$, and each $\Sigma_{k,t}$ is also a minimizer of the weighted area functional in $\Sigma_{k-1}$.

\begin{lemma}\label{induction;lem}
	Let $k\in \{1,\dots,m\}$. Suppose that $\Sigma_k$ is a minimizer of the weighted area 
		\[	\mathcal{H}_{\rho_{k-1}}^{n-k}(\Sigma) = \int_{\Sigma}\rho_{k-1}\, d\mu\]
	 satisfying $R_{\Sigma_k} = (n-m)(n-m-1)$ and $D_{\Sigma_k}\log\rho_k = 0$ on $\Sigma_k$, then the followings hold:
	\begin{enumerate}[(i)]
		\item $A_{\Sigma_k} = 0$, so that $\Sigma_k$ is a totally geodesic hypersurface in $\Sigma_{k-1}$.
		\item $D_{\Sigma_{k-1}}\log\rho_{k-1} = 0$ on $\Sigma_k$.
		\item $R_{\Sigma_{k-1}} - 2\Delta_{\Sigma_{k-1}}\log\rho_{k-1} - |D_{\Sigma_{k-1}}\log\rho_{k-1}|^2 = (n-m)(n-m-1)$ on $\Sigma_k$.
	\end{enumerate}
\end{lemma}

\begin{proof}
	Firstly, taking $f = \rho_{k-1}^{-1/2}$ in the stability inequality (\ref{slicing;eqn3}), we have
	
	\begin{align*}
		0 \leq& 	\int_{\Sigma_k}\rho_{k-1}|D_{\Sigma_k}\rho_{k-1}^{-1/2}|^2 - \frac{1}{2}\int_{\Sigma_k} |A_{\Sigma_k}|^2 - \int_{\Sigma_k} \left(\Delta_{\Sigma_{k}}\log\rho_{k-1} + \frac{1}{2} |D_{\Sigma_{k}}\log\rho_{k-1}|^2\right)  \\
		 = & -\frac{1}{4}\int_{\Sigma_k}|D_{\Sigma_k}\log\rho_{k-1}|^2 - \frac{1}{2}\int_{\Sigma_k} |A_{\Sigma_k}|^2,
	\end{align*}
	which gives $D_{\Sigma_k}\log\rho_{k-1} = 0$ and $A_{\Sigma_k} = 0$ on $\Sigma_k$. This proves (i). Then (ii) follows from $	D_{\Sigma_k}\log\rho_{k-1} =  0$ and
	\[ \langle D_{\Sigma_{k-1}}\log\rho_{k-1},\, \nu_{\Sigma_k}\rangle = H_{\Sigma_k} + \langle D_{\Sigma_{k-1}}\log\rho_{k-1},\, \nu_{\Sigma_k}\rangle = 0.\]
Next, to prove (iii), we note that Corollary \ref{slicing;cor1} gives
\begin{align}\label{induction;eqn1}
	& R_{\Sigma_{k-1}} - 2\Delta_{\Sigma_{k-1}}\log\rho_{k-1} - |D_{\Sigma_{k-1}}\log\rho_{k-1}|^2\\
	 \notag\geq\, & R_M - 2\Delta_{M}\psi - |D_{\Sigma_{M}}\psi|^2\\
	 \notag\geq\, & (n-m)(n-m-1)	
\end{align}
at every point on $\Sigma_k$. On the other hand, Lemma \ref{slicing;lem2} and the assumption give
\begin{align}\label{induction;eqn2}
		&R_{\Sigma_{k-1}} - 2\Delta_{\Sigma_{k-1}}\log\rho_{k-1} - |D_{\Sigma_{k-1}}\log\rho_{k-1}|^2 - (n-m)(n-m-1)\\
		\notag=\, &  R_{\Sigma_{k-1}}  - R_{\Sigma_k} - 2\Delta_{\Sigma_{k-1}}\log\rho_{k-1} - |D_{\Sigma_{k-1}}\log\rho_{k-1}|^2\\
		\notag\leq\, & - 2\Delta_{\Sigma_k}\log\rho_k -  |D_{\Sigma_{k}}\log\rho_{k}|^2 -  |D_{\Sigma_{k}}\log u_{k}|^2\\
		\notag\leq\, &  0
\end{align}
at every point on $\Sigma_k$. Putting (\ref{induction;eqn1}) and (\ref{induction;eqn2}) together, we thus obtain (iii).

\end{proof}

\begin{lemma}\label{foliation;lem}
	Let $k\in \{1,\dots, m\}$. If $R_{\Sigma_k} = (n-m)(n-m-1)$, then there exists a local foliation $\{\Sigma_{k,t}\}_{t\in (-\epsilon, \epsilon)}$ of $\Sigma_k$ in $\Sigma_{k-1}$ such that each $\Sigma_{k,t}$ is given by the graph over $\Sigma_k$ with graph function $w_t$ along the unit normal $\nu_{\Sigma_k}$ such that
	\begin{align*}
		\Sigma_{k,0} = \Sigma_k,\quad \frac{\partial}{\partial t}\Big|_{t=0} w_t = 1,\quad \frac{\partial w_t}{\partial t} > 0,\quad \fint_{\Sigma_k} w_t d\mu = t
	\end{align*}
	and $H_{\Sigma_{k,t}} +\langle D_{\Sigma_{k-1}}\log\rho_{k-1},\, \nu_{\Sigma_{k,t}}\rangle $ is constant on $\Sigma_{k,t}$, where $\nu_{\Sigma_{k,t}}$ is unit normal on $\Sigma_{k,t}$.
\end{lemma}

\begin{proof}
	Denote by $\mathring{C}^{\alpha}(\Sigma_k)$ the space of functions $f\in C^{\alpha}(\Sigma_k)$ with $\int_{\Sigma_k}f d\mu = 0$. For $f\in C^{2,\alpha}(\Sigma_k)$, denote by $\Sigma_f$ the graph of $f$ over $\Sigma_k$, and $\nu_{\Sigma_f}$ the unit normal of $\Sigma_f$. Moreover, denote $\tilde{H}_f = H_{\Sigma_f}  +\langle D_{\Sigma_{k-1}}\log\rho_{k-1},\, \nu_{\Sigma_{f}}\rangle$. We consider the map $\Psi:C^{2,\alpha}(\Sigma_k)\to \mathring{C}^{\alpha}(\Sigma_k)\times \mathbb{R}$ defined by
	
	\[	\Psi(f) = \left(\tilde{H}_f - \fint_{\Sigma_k}\tilde{H}_f\, d\mu,\,  \fint_{\Sigma_k} f\, d\mu \right).\]
	Using Lemma \ref{induction;lem} and the second variation formula, we compute
	
	\begin{align*}
		& \frac{d}{ds}\Big|_{s=0}\tilde{H}_{s\eta} \\
		=\, & -\Delta_{\Sigma_k}\eta - (|A_{\Sigma_k}|^2 + \Ric_{\Sigma_{k-1}}(\nu_{\Sigma_k},\nu_{\Sigma_k})) \eta + (D_{\Sigma_{k-1}}^2 \log\rho_{k-1})(\nu_{\Sigma_k}, \nu_{\Sigma_k}) \eta\\
		=\, & -\Delta_{\Sigma_k}\eta -\frac{1}{2}\left(R_{\Sigma_{k-1}} - R_{\Sigma_k}  -2\Delta_{\Sigma_{k-1}}\log\rho_{k-1} -  |D_{\Sigma_{k-1}}\log\rho_{k-1}|^2
		 \right)\eta\\
		=\, & -\Delta_{\Sigma_k}\eta.
	\end{align*}
This gives

	\begin{align*}
		 D\Psi|_{f=0}(\eta) =  \left(  -\Delta_{\Sigma_k}\eta,\, \fint_{\Sigma_k} \eta\, d\mu\right)\in \mathring{C}^{\alpha}(\Sigma_k)\times \mathbb{R}. 
	\end{align*}
Hence the linearized operator of $\Psi$ at $f=0$ is invertible. Applying the inverse function theorem, we thus find a family of functions $w_t:\Sigma_k\to\mathbb{R}$ for $t\in (-\epsilon, \epsilon)$ such that
	\[	\Psi(w_t) = (0, t),\ w_0 = 0,\ \frac{\partial}{\partial t}\Big|_{t=0} w_t = 1,\ \text{and}\, \fint_{\Sigma_k}w_t\, d\mu = t. \]
Moreover, we can make $\frac{\partial w_t}{\partial t} > 0$ everywhere by taking $\epsilon$ to be sufficiently small. Next,  we denote by $\Sigma_{k,t}$ the graph of $w_t$ over $\Sigma_k$, thus $H_{\Sigma_{k,t}} +\langle D_{\Sigma_{k-1}}\log\rho_{k-1},\, \nu_{\Sigma_{k,t}}\rangle $ is constant on $\Sigma_{k,t}$.
\end{proof}

\bigskip

Next, we show that each slice $\Sigma_{k,t}$ given in Lemma \ref{foliation;lem} are all minimizers of the weighted area functional. To achieve our goal, we adapt the arguments in \cite[Proposition 3.4]{Zhu20} and \cite[Proposition 3.1]{CKL24} to our case.

\begin{proposition}\label{induction;isometry;prop1}
	Let $k\in \{1,\dots, m\}$. If $R_{\Sigma_k} = (n-m)(n-m-1)$ and $\Sigma_{k,t}$ is a slice given in Lemma \ref{foliation;lem}, we have 
		\[	\mathcal{H}_{\rho_{k-1}}^{n-k}(\Sigma_{k,t}) = \mathcal{H}_{\rho_{k-1}}^{n-k}(\Sigma_{k}) = \int_{\Sigma_k}\rho_{k-1}\, d\mu.\]
\end{proposition}

\begin{proof}

Since $\Sigma_{k,t}$ is homologous to $\Sigma_k$, and $\Sigma_k$ is the minimizer of the weighted area among its homology class, we have

\begin{align*}
	0 &\leq \int_{\Sigma_{k,t}}\rho_{k-1}\, d\mu - \int_{\Sigma_{k,0}}\rho_{k-1}\, d\mu\\
	&\leq \int_0^t\int_{\Sigma_{k,s}} \rho_{k-1} w_s \left( H_{\Sigma_{k,s}} + \langle D_{\Sigma_{k-1}}\log\rho_{k-1},\, \nu_{\Sigma_{k,s}}\rangle \right)\, d\mu\, ds.
\end{align*}
Hence to prove Proposition \ref{induction;isometry;prop1}, it suffices to show that
	\begin{align}\label{isometry;eqn1}
		H_{\Sigma_{k,t}} + \langle D_{\Sigma_{k-1}}\log\rho_{k-1},\, \nu_{\Sigma_{k,t}}\rangle \leq 0
	\end{align}
	for all $t\in (0,\epsilon)$. Suppose in contrary that there exists $t_0\in (0,\epsilon)$ and $\delta >0$ such that
	\[	H_{\Sigma_{k,t_0}} + \langle D_{\Sigma_{k-1}}\log\rho_{k-1},\, \nu_{\Sigma_{k,t_0}}\rangle > 2\delta.\]
Note that the hypersurfaces $\Sigma_k$ and $\Sigma_{k,t_0}$ are non-intersecting. We consider the brane functional
	\begin{align}
		\mathcal{B}(\hat{\Omega}) = \int_{\partial\hat{\Omega}\setminus\Sigma_k}\rho_{k-1}\, d\mu - \delta\int_{\hat{\Omega}}\rho_{k-1}\, d\mu,,
	\end{align}
for Borel subsets $\hat{\Omega}$ of the region between $\Sigma_k$ and $\Sigma_{k,t_0}$, with finite perimeter and $\Sigma_k\subset\partial\hat{\Omega}$. Since
	\[	H_{\Sigma_{k}} + \langle D_{\Sigma_{k-1}}\log\rho_{k-1},\, \nu_{\Sigma_{k}}\rangle <\delta < H_{\Sigma_{k,t_0}} + \langle D_{\Sigma_{k-1}}\log\rho_{k-1},\, \nu_{\Sigma_{k,t_0}}\rangle,\]
	the hypersurfaces $\Sigma_k$ and $\Sigma_{k,t_0}$ serve as barriers. Consequently we can find a Borel set $\hat{\Omega}$ which is the minimizer of $\mathcal{B}$, such that $\partial\hat{\Omega}\setminus \Sigma_k$ is a smooth two-sided hypersurface disjoint from $\Sigma_k$ and $\Sigma_{k,t_0}$. 

Next, recall from Proposition \ref{slicing;exist;prop} that there is a smooth 1-Lipschitz map $\Phi_k: (\Sigma_k, g|_{\Sigma_k})\to (S^{n-m}\times\mathbb{T}^{m-k}, g_{S^{n-m}} + g_{\mathbb{T}^{m-k}})$ of non-zero degree, such that $\Phi_k = \Phi_{k-1}|_{\Sigma_k}$ and $\deg(\Phi_k) = \deg(\Phi_{k-1})$. Since $\partial\hat{\Omega}\setminus \Sigma_k$ is homologous to $\Sigma_k$, Stoke's theorem implies that the map $ \Phi_{k-1}|_{\partial\hat{\Omega}\setminus \Sigma_k}$ has non-zero degree. We take the connected component $\hat{\Sigma}_k$ of $\partial\hat{\Omega}\setminus \Sigma_k$ such that the map $\hat{\Phi}_k = \Phi_{k-1}|_{\hat{\Sigma}_k}: (\hat{\Sigma}_k, g|_{\hat{\Sigma}_k})\to (S^{n-m}\times\mathbb{T}^{m-k}, g_{S^{n-m}} + g_{\mathbb{T}^{m-k}})$ has non-zero degree.

Let $\hat{u}_k$ be a first eigenfunction of the stability operator on $\hat{\Sigma}_k$, and let $\hat{\rho}_k = \hat{u}_k\rho_{k-1}$. We now apply Proposition \ref{slicing;exist;prop} to $(\hat{\Sigma}_k, \hat{\rho}_k)$, and find a stable weighted slicing 
	\[	\hat{\Sigma}_m\subset\cdots\subset\hat{\Sigma}_k\]
with smooth 1-Lipschitz maps $\hat{\Phi}_j: (\hat{\Sigma}_j, g|_{\hat{\Sigma}_j}) \to (S^{n-m}\times\mathbb{T}^{m-j}, g_{S^{n-m}} + g_{\mathbb{T}^{m-j}})$ of non-zero degree, for $j\in \{k,\dots, m\}$.

\begin{claim}\label{induction;isometry;lem1}
	The function $\hat{\rho}_k$ satisfies the inequality	
	\begin{align*}
	&\Delta_{\hat{\Sigma}_k}\log\hat{\rho}_k + \frac{1}{2} |D_{\hat{\Sigma}_{k}}\log\hat{\rho}_{k}|^2	\\
	\leq\, & \Delta_{\Sigma_{k-1}}\log\rho_{k-1} + \frac{1}{2} |D_{\Sigma_{k-1}}\log\rho_{k-1}|^2 -\frac{1}{2}(R_{\Sigma_{k-1}} - R_{\hat{\Sigma}_k} + |A_{\hat{\Sigma}_k}|^2 + \delta^2) - \frac{1}{2} |D_{\hat{\Sigma}_{k}}\log \hat{u}_{k}|^2.
	\end{align*}
\end{claim}

\begin{proof}[Proof of Claim \ref{induction;isometry;lem1}]
	The proof is similar to the proof of Lemma \ref{slicing;lem2} as in \cite{BH24, SY79, SY17}. Recall that $\hat{u}_k$ is a first eigenfunction of the stability operator of $\mathcal{B}$, thus 
	
	\begin{align*}
	0 &\leq -\Delta_{\hat{\Sigma}_k}\hat{u}_k - \langle D_{\hat{\Sigma}_k}\rho_{k-1},\, 	D_{\hat{\Sigma}_k}\log\hat{u}_k\rangle - (\Ric_{\Sigma_{k-1}}(\nu_{\hat{\Sigma}_k},\,\nu_{\hat{\Sigma}_k}) + |A_{\hat{\Sigma}_k}|^2)\hat{u}_k\\
	&\quad + (D_{\Sigma_{k-1}}^2\log\rho_{k-1})( \nu_{\hat{\Sigma}_k},\,\nu_{\hat{\Sigma}_k})\hat{u}_k.
	\end{align*}
Using the Gauss equation and $H_{\hat{\Sigma}_k} + \langle D_{\Sigma_{k-1}}\log\rho_{k-1},\, \nu_{\hat{\Sigma}_k}\rangle = \delta$, we obtain

\begin{align*}
	0 &\leq -\Delta_{\hat{\Sigma}_k}\log\hat{u}_k - 	|D_{\hat{\Sigma}_{k}}\log\hat{u}_{k}|^2 - \langle D_{\hat{\Sigma}_k}\log\rho_{k-1},\, 	D_{\hat{\Sigma}_k}\log\hat{u}_k\rangle\\
	&\quad - \frac{1}{2} (R_{\Sigma_{k-1}} - R_{\hat{\Sigma}_k}  + |A_{\hat{\Sigma}_k}|^2 + \delta^2)\\
	&\quad + \Delta_{\Sigma_{k-1}}\log\rho_{k-1} - \Delta_{\hat{\Sigma}_k}\log\rho_{k-1} + \frac{1}{2} \langle D_{\Sigma_{k-1}}\log\rho_{k-1},\, \nu_{\hat{\Sigma}_k}\rangle^2\\
	&= -\Delta_{\hat{\Sigma}_k}\log\hat{\rho}_k - 	\frac{1}{2}|D_{\hat{\Sigma}_{k}}\log\hat{\rho}_{k}|^2 - \frac{1}{2} (R_{\Sigma_{k-1}} - R_{\hat{\Sigma}_k} +  |A_{\hat{\Sigma}_k}|^2 + \delta^2)\\
	&\quad +\Delta_{\Sigma_{k-1}}\log\rho_{k-1} + \frac{1}{2}|D_{\Sigma_{k-1}}\log\rho_{k-1}|^2 - \frac{1}{2} |D_{\hat{\Sigma}_{k}}\log\hat{u}_{k}|^2,
\end{align*}
where we have used $\langle D_{\Sigma_{k-1}}\log\rho_{k-1},\, \nu_{\hat{\Sigma}_k}\rangle^2 = |D_{\Sigma_{k-1}}\log\rho_{k-1}|^2 - |D_{\hat{\Sigma}_{k}}\log\rho_{k-1}|^2$ in the last step. 
\end{proof}

Next, applying Proposition \ref{slicing;prop1} to the slicing
	\[	\hat{\Sigma}_m\subset\cdots\subset\hat{\Sigma}_k,\]
we get

\begin{align}\label{isometry;eqn2}
0 \leq& 	\int_{\hat{\Sigma}_m}\hat{\rho}_{m-1}|D_{\hat{\Sigma}_m}f|^2  - \frac{1}{2}\int_{\hat{\Sigma}_m}(R_{\hat{\Sigma}_k} - R_{\hat{\Sigma}_m} + |A_{\hat{\Sigma}_m}|^2)\hat{\rho}_{m-1}f^2\\
		&\notag + \int_{\hat{\Sigma}_m} \left(\Delta_{\hat{\Sigma}_k}\log\hat{\rho}_{k} + \frac{1}{2} |D_{\hat{\Sigma}_k}\log\hat{\rho}_{k}|^2
		 \right) \hat{\rho}_{m-1}f^2 \\
		&\notag - \int_{\hat{\Sigma}_m} \left(\Delta_{\hat{\Sigma}_m}\log\hat{\rho}_{m-1} + \frac{1}{2} |D_{\hat{\Sigma}_m}\log\hat{\rho}_{m-1}|^2
		 \right) \hat{\rho}_{m-1}f^2	
\end{align}
for all $f\in C^{\infty}(\hat{\Sigma}_m)$. Applying Claim \ref{induction;isometry;lem1} and Corollary \ref{slicing;cor1} to (\ref{isometry;eqn2}), we obtain

\begin{align}\label{isometry;eqn3}
	0 \leq& 	\int_{\hat{\Sigma}_m}\hat{\rho}_{m-1}|D_{\hat{\Sigma}_m}f|^2  - \frac{1}{2}\int_{\hat{\Sigma}_m}((n-m)(n-m-1) - R_{\hat{\Sigma}_m} + \delta^2)\hat{\rho}_{m-1}f^2\\
		&\notag - \int_{\hat{\Sigma}_m} \left(\Delta_{\hat{\Sigma}_m}\log\hat{\rho}_{m-1} + \frac{1}{2} |D_{\hat{\Sigma}_m}\log\hat{\rho}_{m-1}|^2
		 \right) \hat{\rho}_{m-1}f^2	
\end{align}
for all $f\in C^{\infty}(\hat{\Sigma}_m)$. Now, recall that on the slice $\hat{\Sigma}_m^{n-m}$ there is a 1-Lipschitz map $\hat{\Phi}_m:\hat{\Sigma}_m\to S^{n-m}$ of non-zero degree. We can repeat the argument in Section \ref{spectral;Llarull;section}, but with the assumption (\ref{spectral;Llarull;assumption}) replaced by (\ref{isometry;eqn3}). This leads to

\begin{align*}
	0\leq 	-2\int_{\hat{\Sigma}_m}\left( \frac{(n-m)(n-m-1)}{4}|s|^2 + \langle \R^Es,\, s\rangle + \frac{\delta^2}{4}\right) \leq -\frac{\delta^2}{2}\vol(\hat{\Sigma}_m),
\end{align*}
a contradiction. This finishes the proof of  Proposition \ref{induction;isometry;prop1}.
\end{proof}

\bigskip

\section{Proof of Main Theorems}

We are now ready to assemble the ingredients from the previous sections. Throughout we keep the notation and slicing constructed in Proposition \ref{slicing;exist;prop}.

\subsection{Proof of Theorem \ref{main;thm;1}}

Consider any stable weighted slicing in $M$ of order $m$ as constructed in Proposition \ref{slicing;exist;prop}. On the bottom slice $\Sigma_m$, the stability inequality (\ref{slicing;eqn4}) implies
		\begin{align*}
	0 \leq& \int_{\Sigma_m}\rho_{m-1}|D_{\Sigma_m}f|^2  - \frac{1}{2}\int_{\Sigma_m} \left(  (n-m)(n-m-1) -  R_{\Sigma_m}\right) \rho_{m-1}f^2\\
		&\notag - \int_{\Sigma_m} \left(\Delta_{\Sigma_{m}}\log\rho_{m-1} + \frac{1}{2} |D_{\Sigma_{m}}\log\rho_{m-1} |^2
		 \right) \rho_{m-1}f^2
	\end{align*}
for all $f\in C^{\infty}(\Sigma_m)$. On the other hand, from Proposition \ref{slicing;exist;prop} there is a 1-Lipschitz map $\Phi_m:(\Sigma_m, g_{\Sigma_m}) \to (S^{n-m}, g_{S^{n-m}})$ of non-zero degree.  Therefore, it follows from Theorem \ref{spectral;Llarull;thm} that  $(\Sigma_m, g_{\Sigma_m})$ is isometric to $(S^{n-m}, g_{S^{n-m}})$ and the scalar curvature satisfies  $R_{\Sigma_m} = (n-m)(n-m-1)$ for the bottom slice $\Sigma_m$ in all stable weighted slicings of order $m$.

We now propagate the rigidity of the bottom slice through the entire slicing: the foliations of each slice built in Section \ref{foliation;section} allow us to carry this isometry upward, slice by slice, until it extends to the whole manifold $M$.

\begin{proposition}[Slice-by-slice rigidity]\label{induction;isometry;prop2}
Assume the hypotheses of Theorem~\ref{main;thm;1}.  
Let 
\[
  \Sigma_m \subset \Sigma_{m-1} \subset \cdots \subset \Sigma_1 \subset 
  \Sigma_0 = M
\]
be any stable weighted slicing of order \(m\) constructed in
Proposition~\ref{slicing;exist;prop}.  Then, for every 
\(k\in\{0,1,\dots,m\}\), we have $R_{\Sigma_k} = (n-m)(n-m-1)$ and $\rho_{k-1}$ is constant on the slice $\Sigma_k$. In addition, $\Sigma_{k-1}$  is isometrically covered by $\Sigma_k\times\mathbb{R}$.
\end{proposition}

\begin{proof}
	 We prove by induction on $k$. The base case has been already settled, so that $R_{\Sigma_m} = (n-m)(n-m-1)$ for all $\Sigma_m$ in all stable weighted slicing of order $m$. Now assume that all  $k$-th order slice $\Sigma_k$ satisfy $R_{\Sigma_k} = (n-m)(n-m-1)$, we want to prove that for every $(k-1)$-th order slice $\Sigma_{k-1}$, we have $R_{\Sigma_{k-1}} = (n-m)(n-m-1)$ and $\Sigma_{k-1}$ is isometrically covered by $\Sigma_k\times\mathbb{R}$.
	
	From Lemma \ref{foliation;lem} and Proposition \ref{induction;isometry;prop1}, we can find a local foliation $\{\Sigma_{k,t}\}_{t\in (-\epsilon, \epsilon)}$ of $\Sigma_k$ in $\Sigma_{k-1}$ such that each $\Sigma_{k,t}$ is also a minimizer of the weighted area. Hence $\Sigma_{k,t}$ is also a $k$-th order slice in a stable weighted slicing of the form	
	\[	\Sigma_{m,t}\subset\cdots\subset\Sigma_{k,t}\subset \Sigma_{k-1}\subset\cdots\subset \Sigma_0=M^n\]
	 for all $t\in (-\epsilon, \epsilon)$. If follows from the induction hypothesis that $\Sigma_{k,t}$ also satisfies	 
	 \begin{align}\label{isometry;eqn4}
	 	R_{\Sigma_{k,t}} = (n-m)(n-m-1)
	 \end{align}
for all $t\in (-\epsilon, \epsilon)$. From Lemma \ref{induction;lem}, we have	 
	 \[	A_{\Sigma_{k,t}} = 0,\quad\text{and}\quad D_{\Sigma_{k-1}}\log\rho_{k-1} = 0\]
	 on the local foliation $\{\Sigma_{k,t}\}_{t\in (-\epsilon, \epsilon)}$. This together with Lemma \ref{induction;lem}(iii) imply	 
	 \begin{align}\label{isometry;eqn5}
	 	R_{\Sigma_{k-1}}  = (n-m)(n-m-1)
	 \end{align}
	 on the local foliation $\{\Sigma_{k,t}\}_{t\in (-\epsilon, \epsilon)}$. On the other hand, we can write	 
	 	\[	g_{\Sigma_{k-1}} = \phi^2dt^2 + g_{\Sigma_{k,t}}\]
	 on the local foliation $\{\Sigma_{k,t}\}_{t\in (-\epsilon, \epsilon)}$, where $\phi$ is the lapse function. Combining (\ref{isometry;eqn4}) and (\ref{isometry;eqn5}), we obtain	 
	 \begin{align*}
	 	(n-m)(n-m-1) &= R_{\Sigma_{k-1}} \\
	 	&= R_{\Sigma_{k,t}} - 2\phi^{-1}\Delta_{\Sigma_{k,t}}\phi\\
	 	&= (n-m)(n-m-1) - 2\phi^{-1}\Delta_{\Sigma_{k,t}}\phi,
	 \end{align*}
	 which implies	 
	 \begin{align}
	 	\Delta_{\Sigma_{k,t}}\phi = 0
	 \end{align}
	 on the local foliation $\{\Sigma_{k,t}\}_{t\in (-\epsilon, \epsilon)}$. Consequently, $\phi = \phi(t)$ is a function depending only on $t$. Lemma \ref{foliation;lem} implies $\phi(0) = 0$. In addition, $A_{\Sigma_{k,t}} = 0$ implies $\partial_t g_{\Sigma_{k,t}} = 0$. Consequently,
	 	\[	g_{\Sigma_{k-1}} = \phi^2dt^2 + g_{\Sigma_{k}}\] 
	 on the local foliation $\{\Sigma_{k,t}\}_{t\in (-\epsilon, \epsilon)}$. Using a continuity argument as in \cite[Proposition 11]{BBN10}, we  conclude that $\rho_{k-1}$ is constant and $\Sigma_{k-1}$ is isometrically covered by $\Sigma_k\times\mathbb{R}$.
\end{proof}

\medskip

To finish the proof of Theorem \ref{main;thm;1}, we apply Proposition \ref{induction;isometry;prop2} with the base case that $(\Sigma_m, g_{\Sigma_m})$ is isometric to $(S^{n-m}, g_{S^{n-m}})$. Therefore, $(M,g)$ is isometrically covered by $(S^{n-m}\times\mathbb{R}^m,\, g_{S^{n-m}} + g_{\mathbb{R}^m})$ and $\psi$ is a constant function.

\bigskip

\subsection{Proof of Theorem \ref{main;thm;2}}
In the first step, we construct a $\mu$-bubble that allows us to apply Theorem \ref{main;thm;1}. Denote the projection of $\Phi$ onto the factors by $\phi:M\to S^{n-m}\times \T^m$ and  $\varphi:M\to [-1,1]$.	 By assumption, the map $\pr_{S^{n-m}}\circ\phi = \pr_{S^{n-m}}\circ\Phi$ is 1-Lipschitz. Let $\Theta$ be a top-form of $S^{n-m}\times\T^m$ such that $\int_{S^{n-m}\times\T^m}\Theta = 1$. Define the pull-back form $\omega = \phi^*\Theta$. By Sard's theorem, we can find a regular value $t_0\in (-1,1)$ of the map $\varphi$. Let $\hat{\Sigma} = \varphi^{-1}(t_0)$. Because $M$ is connected and $\varphi$ is continuous, and by assumption $\Phi(\partial_{-}M)\subset S^{n-m}\times \T^m \times\{-1\}$ and $\Phi(\partial_{+}M)\subset S^{n-m}\times \T^m \times\{1\}$, we see that  $\varphi(M)$ is a connected subset of $[-1,1]$ containing both $-1$ and $1$. It follows that $\varphi(M) = [-1,1]$ and $ \varphi^{-1}(t)$ is non-empty for all regular values $t$. So $\hat{\Sigma}$ is a smooth, orientable and embedded hypersurface in $M$. By the coarea formula,
\begin{align}\label{bandwidth;eqn1}
	\deg(\Phi) = \int_M	d\varphi\wedge\omega = \int_{-1}^1\left(\int_{\varphi^{-1}(t)}\omega\right) dt
\end{align}
On the other hand, if $t_1 < t_2$ are two regular values, the Stokes theorem gives
\begin{align}\label{bandwidth;eqn2}
	0 = \int_{\varphi^{-1}([t_1, t_2])} d\omega = \int_{\varphi^{-1}(t_2)}\omega - \int_{\varphi^{-1}(t_1)}\omega.
\end{align}
Putting (\ref{bandwidth;eqn1}) and (\ref{bandwidth;eqn2}) together, we obtain
\begin{align}\label{bandwidth;eqn3}
	\int_{\hat{\Sigma}}\omega = \deg(\Phi).
\end{align}

Now, suppose in contrary that
\[	d(\partial_-M, \partial_+M) > \frac{2\pi}{\sqrt{\sigma}} := L.\]
We can find $\epsilon>0$ and a regular value $t_0\in (0,1)$ of $\varphi$ such that 
\[	d(\partial_-M, \hat{\Sigma}) \geq  L+\epsilon ,\]
where $\hat{\Sigma} = \varphi^{-1}(t_0)$. We then find a smooth function $\beta:\varphi^{-1}([0, t_0])\to (0,L)$
	such that $|\text{Lip}\, \beta|\leq 1-\epsilon$, $\beta\to 0$ at $\partial_{-}M$ and $\beta\to L$ at $\hat{\Sigma}$. Consider the function $h:\varphi^{-1}([0, t_0])\to\mathbb{R}$ given by
	\[	h(x) = -2(1-\epsilon)\frac{\pi}{L}\tan\left(\frac{\pi}{L}\beta(x) - \frac{\pi}{2} \right).\]
Observe that
	\begin{align}\label{bandwidth;eqn4}
		&\frac{1}{2}h(x)^2 - |\nabla_M h|(x)\\
		\notag =&\, 2(1-\epsilon)^2\frac{\pi^2}{L^2}\tan^2\left(\frac{\pi}{L}\beta(x) - \frac{\pi}{2} \right) - 2(1-\epsilon)\frac{\pi^2}{L^2}\sec^2\left(\frac{\pi}{L}\beta(x) - \frac{\pi}{2} \right)|\nabla\beta|\\
		\notag \geq&\, -(1-\epsilon)^2\, \frac{\sigma}{2}.
	\end{align}

Next, let  $\Omega_0$ be a reference Caccioppoli set in $\varphi^{-1}([0, t_0])$ such that $\hat{\Sigma}\subset\Omega_0$. By \cite{CL24}, we can find a warped $\mu$-bubble $\Omega$ minimizing
	\[\mathcal{A}(\Omega) = \int_{\partial^*\Omega}e^{\psi} d\mathcal{H}^{n-1} - \int_{\Omega}(\chi_{\Omega}-\chi_{\Omega_0})h e^{\psi} \, d\mathcal{H}^n\]
among all Caccioppoli sets $\tilde{\Omega}$ in $\varphi^{-1}([0, t_0])$ with $\tilde{\Omega}\Delta\Omega_0\subset \varphi^{-1}((0, t_0))$. Then $\partial\Omega\setminus\hat{\Sigma}$ is a smooth closed, two-sided hypersurface disjoint from $\partial_-M$ and $\hat{\Sigma}$. By Stoke's theorem and (\ref{bandwidth;eqn3}), 
\begin{align*}
	\deg(\phi|_{\partial\Omega\setminus\hat{\Sigma}}) = \int_{\partial\Omega\setminus\hat{\Sigma}}\omega = 	\int_{\hat{\Sigma}}\omega = \deg(\Phi) \neq 0.
\end{align*}
We take $\Sigma$ to be the connected component of $\partial\Omega\setminus\hat{\Sigma}$  such that the map $\Psi:\Sigma\to S^{n-m}\times\T^m$ defined by $\Psi := \phi|_{\Sigma}$ has $\deg(\Psi)\neq 0$. 

To summarize, we obtain a $\mu$-bubble $\Sigma\subset M$ and a smooth map $\Psi:\Sigma\to S^{n-m}\times\T^m$ such that $\deg(\Psi)\neq 0$ and $\pr_{S^{n-m}}\circ\Psi$ is 1-Lipschitz. From first and the second variation formulae of wrapped $\mu$-bubble \cite{CL24}, we have
	\begin{align*}
		H_{\Sigma} &= -e^{-\psi}\langle D_M e^{\psi},\nu_{\Sigma}\rangle + h = -\langle D_M \psi,\nu_{\Sigma}\rangle + h,
	\end{align*}
and
\begin{align*}
	0 &\leq -\Delta_\Sigma f  - (\Ric_M(\nu_\Sigma,\, \nu_\Sigma)+ |A_{\Sigma}|^2)f +(D_M^2 \log e^{\psi})(\nu_{\Sigma},\, \nu_{\Sigma})f  \\
	&\quad  - \langle D_{\Sigma} \log e^{\psi},\, D_{\Sigma}f\rangle - \langle D_Mh,\, \nu_{\Sigma}\rangle f
\end{align*}
for all $f\in C^{\infty}(\Sigma)$. Hence, we can find a positive function $u\in C^{\infty}(\Sigma)$ such that 

\begin{align*}
	0 &\leq  -\Delta_{\Sigma}\log u - |D_{\Sigma}\log u|^2  - 	\frac{1}{2}(R_M - R_{\Sigma}+ |A_{\Sigma}|^2 + H_{\Sigma}^2)\\
	&\quad +(D_M^2 \psi)(\nu_{\Sigma},\, \nu_{\Sigma}) - \langle D_{\Sigma} \psi,\, D_{\Sigma}\log u\rangle - \langle D_Mh,\, \nu_{\Sigma}\rangle\\
	&\leq -\Delta_{\Sigma}\log u - |D_{\Sigma}\log u|^2  - 	\frac{1}{2}(R_M - R_{\Sigma} + |A_{\Sigma}|^2) - \frac{1}{2}H_{\Sigma}^2 \\
	&\quad + \Delta_M\psi - \Delta_{\Sigma}\psi - H_{\Sigma}\langle D_M\psi,\,\nu_{\Sigma}\rangle - \langle D_{\Sigma} \psi,\, D_{\Sigma}\log u\rangle - \langle D_Mh,\, \nu_{\Sigma}\rangle.
\end{align*}
Define $\rho = u e^{\psi}$. From the above we obtain

\begin{align}\label{bandwidth;eqn5}
	0 &\leq -\Delta_{\Sigma}\log\rho -\frac{1}{2}|D_{\Sigma}\log\rho|^2 - 	\frac{1}{2}(R_M - R_{\Sigma}+ |A_{\Sigma}|^2) - \frac{1}{2}H_{\Sigma}^2\\
	\notag &\quad + \Delta_M\psi + \frac{1}{2}(|D_M\psi|^2 - \langle D_M\psi,\,\nu_{\Sigma}\rangle^2) - H_{\Sigma}\langle D_M\psi,\,\nu_{\Sigma}\rangle - \langle D_Mh,\, \nu_{\Sigma}\rangle\\
	\notag &\leq -\Delta_{\Sigma}\log\rho -\frac{1}{2}|D_{\Sigma}\log\rho|^2 - 	\frac{1}{2}(R_M - R_{\Sigma}+ |A_{\Sigma}|^2) + \Delta_M\psi + \frac{1}{2}|D_M\psi|^2\\
	\notag &\quad - \left(\frac{1}{2}h^2 - | D_Mh|\right).
\end{align}
Applying (\ref{bandwidth;eqn4}) and the assumption
\[ -\Delta_M\psi -\frac{1}{2}|D_M\psi|^2 + \frac{1}{2}\Big(R_M - (n-m)(n-m-1) - \sigma \Big) \geq 0,\]
we thus obtain
\begin{align}
	0 &\leq 	-\Delta_{\Sigma}\log\rho -\frac{1}{2}|D_{\Sigma}\log\rho|^2 - 	\frac{1}{2}( (n-m)(n-m-1)  - R_{\Sigma}) \\
	\notag &\quad - \left( 1 - (1-\epsilon)^2 \right)\frac{\sigma}{2}.
\end{align}

In conclusion, we have obtained a smooth map $\Psi:\Sigma^n\to S^{n-m}\times\T^m$ of non-zero degree such that $\pr_{S^{n-m}}\circ\Psi$ is 1-Lipschitz, and a smooth function $\log\rho$ on $\Sigma$ satisfying
\begin{align}\label{bandwidth;eqn6}
	-\Delta_{\Sigma}\log\rho -\frac{1}{2}|D_{\Sigma}\log\rho|^2 + 	\frac{1}{2}(R_{\Sigma} - (n-m)(n-m-1)   ) > 0.
\end{align}
	Theorem \ref{main;thm;1} then implies $R_{\Sigma} = (n-m)(n-m-1)$ and $\rho$ is a constant function, contradicting to (\ref{bandwidth;eqn6}). This finishes the proof of Theorem \ref{main;thm;2}.

\begin{remark}
	We note that if the curvature assumption in Theorem \ref{main;thm;2} is replaced by $R_M\geq (n-m)(n-m-1) + \sigma$, then we can actually obtain a sharper inequality for the bandwidth.
\end{remark}


\bigskip







\begin{thebibliography}{99}

\bibitem{Bar24}
C.~B\"ar, \textit{Dirac eigenvalues and the hyperspherical radius}, arXiv.2407.21704



\bibitem{BBCH24}
C.~B\"ar, S.~Brendle, T.-K.A.~Chow and B.~Hanke, \textit{Rigidity results for initial data sets satisfying the dominant energy condition}, arXiv:2304.04145v3



\bibitem{BBHW24}
C.~B\"ar, S.~Brendle, B.~Hanke and Y.~Wang, \textit{Scalar Curvature Rigidity of Warped Product Metrics}, Symmetry, integrability and geometry, methods and applications, 2024-01

\bibitem{BBN10}
H.~Bray, S.~Brendle and A.~Neves, \textit{Rigidity of area-minimizing two-spheres in three-manifolds}, Communications in analysis and geometry, 2010, Vol.18 (4), p.821-830


\bibitem{BHJ23}
S.~Brendle, S.~Hirsch and F.~Johne, \textit{A generalization of Geroch's conjecture}, Communications on pure and applied mathematics, 2024-01, Vol.77 (1), p.441-456

\bibitem{BH24}
S.~Brendle and P.K.~Hung, \textit{Systolic inequalities and the Horowitz-Myers conjecture}, arXiv:2406.04283


\bibitem{BH25}
S.~Brendle and P.K.~Hung, \textit{The rigidity statement in the Horowitz-Myers conjecture}, arXiv:2504.16812


\bibitem{CHS24}
S.~Cecchini, B.~Hanke and T.~Shick, \textit{Lipschitz rigidity for scalar curvature}, Journal of the
European Mathematical Society (2024)

\bibitem{CZ24}
S.~Cecchini and R.~Zeidler, \textit{Scalar and mean curvature comparison via the Dirac operator}, Geometry and Topology 28 (2024), 1167-1212


\bibitem{CL24}
O.~Chodosh and C.~Li, \textit{Generalized soap bubbles and the topology of manifolds with positive scalar curvature}, Ann. of Math. vol. 199, no. 2, pp. 707-740 (2024)


\bibitem{CKL24}
J.~Chu, K.-K.~Kwong and M.-C.~Lee, \textit{Rigidity on non-negative intermediate curvature}, Mathematical Research Letters, Volume 31 (2024) Number 6, p.1693-1714


\bibitem{CLZ24}
J.~Chu, M.-C.~Lee and J.~Zhu, \textit{Llarull's theorem on punctured sphere with $L^{\infty}$ metric}, arXiv:2405.19724, to appear in Ann. Sc. Norm. Sup. Cl. Sci.



\bibitem{GS02}
S.~Goette and U.~Semmelmann, \textit{Scalar curvature estimates for compact symmetric spaces}, Differential geometry and its applications, 2002, Vol.16 (1), p.65-78


\bibitem{Gro18}
M.~Gromov, \textit{Metric Inequalities with Scalar Curvature}, Geometric and Functional Analysis, vol. 28, 645-726, (2018). 


\bibitem{Gro19}
M.~Gromov, \textit{Four Lectures on scalar curvature}, Perspective in scalar curvature, vol. 1, 1-514. arXiv: 1908.10612


\bibitem{GL}
M.~Gromov and H.B.~Lawson, \textit{Positive scalar curvature and the Dirac operator on complete Riemannian manifolds}, Inst. Hautes Etudes Sci. Publ. Math. (1983), no. 58, 83-196 (1984)


\bibitem{HSS24}
T.~Hao, Y.~Shi and Y.~Sun, \textit{Llarull type theorems on complete manifolds with positive scalar curvature}, Transactions of the American Mathematical Society, 2024-10


\bibitem{HKKZ23}
S.~Hirsch, D.~Kazaras, M.~Khuri and Y.~Zhang, \textit{Spectral Torical Band Inequalities and Generalizations of the Schoen-Yau Black Hole Existence Theorem}, International mathematics research notices, 2024-02, Vol.2024 (4), p.3139-3175



\bibitem{HKKZ25}
S.~Hirsch, D.~Kazaras, M.~Khuri and Y.~Zhang, \textit{ Rigid comparison geometry for Riemannian
bands and open incomplete manifolds}, Mathematische Annalen 391 (2025), 2587-2652.



\bibitem{HLS}
Y.~Hu, P.~Liu and Y.~Shi, \textit{Rigidity of 3D spherical caps via $\mu$-bubbles}, Pacific Journal of
Mathematics 323 (2023), 89-114.



\bibitem{LT22}
M.-C.~Lee and L.-F.~Tam, \textit{ Rigidity of Lipschitz map using harmonic map heat flow}, arXiv:2207.11017, to appear in Amer. J. Math.



\bibitem{Listing10}
M.~Listing, \textit{ Scalar curvature on compact symmetric spaces}, arXiv:1007.1832



\bibitem{Llarull}
M.~Llarull, \textit{Sharp estimates and the Dirac operator}, Mathematische annalen, 1998-01, Vol.310 (1), p.55-71



\bibitem{Lott21}
J.~Lott, \textit{ Index theory for scalar curvature on manifolds with boundary}, Proceedings of the
American Mathematical Society 149 (2021), 4451-4459.



\bibitem{Tony}
T.~Tony, \textit{ Scalar curvature rigidity and the higher mapping degree}, Journal of Functional Analysis, 2025-02, Vol.288(3)



\bibitem{SY79}
S.T.~Yau and R.Schoen, \textit{On the structure of manifolds with positive scalar curvature}, Manuscripta mathematica, 1979-01, Vol.28 (1-3), p.159-183



\bibitem{SY17}
S.T.~Yau and R.Schoen, \textit{Positive scalar curvature and minimal hypersurface singularities,}, arXiv:1704.05490





\bibitem{Zeid}
R. Zeidler, \textit{Band width estimates via the Dirac operator}, Journal of differential geometry, 2022-09, Vol.122 (1)


\bibitem{Zhang20}
W.~Zhang, \textit{Nonnegative scalar curvature and area decreasing maps}, SIGMA Symmetry Integrability Geom. Methods Appl. 16 (2020), Paper No. 033


\bibitem{Zhu20}
J.~Zhu, \textit{Rigidity of area-minimizing $2$-spheres in $n$-manifolds with positive scalar curvature}, Proceedings of the American Mathematical Society, 2020-08, Vol.148 (8), p.3479-3489




\bibitem{Zhu21}
J. Zhu, \textit{Width estimate and doubly warped product}, Transactions of the American Mathematical Society, 2021-02, Vol.374 (2), p.1497-1511


\end{thebibliography}
\end{document}